\documentclass{amsart}
\usepackage{amsmath,amssymb,epsfig,arydshln,float,amsfonts,latexsym,color,algorithm,fancyhdr}

\definecolor{lightblue}{rgb}{0.22,0.45,0.70}

\usepackage{subfigure}
\usepackage{bm}

\newtheorem{lemma}{Lemma}[section]
\newtheorem{theorem}{Theorem}[section]
\newtheorem{proposition}{Proposition}[section]
\newtheorem{corollary}{Corollary}[section]
\newtheorem{remark}{Remark}[section]

\newcommand{\vertiii}[1]{{\left\vert\kern-0.25ex\left\vert\kern-0.25ex\left\vert #1 
    \right\vert\kern-0.25ex\right\vert\kern-0.25ex\right\vert}}

\newcommand{\vdiv}{\operatorname*{div}}

\newcommand{\bu}{\boldsymbol{u}}
\newcommand{\bv}{\boldsymbol{v}}

\newcommand\bH{\mathbf{H}}

\newcommand{\bV}{\mathbf{V}}

\newcommand{\TT}{\mathbb{T}}
\newcommand{\NN}{\mathbb{N}}

\newcommand\cR{\mathcal{R}}
\newcommand\cT{\mathcal{T}}
\newcommand\cX{\mathcal{X}}

\newcommand\cTh{\widehat{\mathcal{T}}}
\numberwithin{equation}{section}
\numberwithin{table}{section}
\numberwithin{figure}{section}
\allowdisplaybreaks

\newcommand\todo[1]{\begin{color}{black}#1\end{color}}

\begin{document}
\title[AFEM for the Stokes Eigenvalue problem]{Adaptive Mixed FEM for the Stokes eigenvalue problem}
\author[D. Boffi]{Daniele Boffi}
\address{Computer, electrical and mathematical sciences and engineering division, King
Abdullah University of Science and Technology, Thuwal 23955, Saudi Arabia and
Dipartimento di Matematica ``F. Casorati'', University of Pavia, Via Ferrata 1, 27100, Pavia, Italy.}
\email{daniele.boffi@kaust.edu.sa}
\thanks{DB is member of the INdAM Research group GNCS.}
\author[A. Khan]{Arbaz Khan}
\address{Department of Mathematics, Indian Institute of Technology Roorkee, Roorkee 247667, India.}
\email{arbaz@ma.iitr.ac.in}
\thanks{Advancements in this project occurred when AK visited DB at Kaust, Saudi Arabia from October 13 to 24, 2022.
The funding for this visit was provided by DB, Kaust, Saudi Arabia. 
 AK received partial support from the Sponsored Research \& Industrial Consultancy (SRIC), Indian Institute of Technology Roorkee, India, through the faculty initiation grant MTD/FIG/100878, as well as from the SERB MATRICS grant MTR/2020/000303 and the SERB Core research grant CRG/2021/002569.}
\date{\today}
\subjclass[2020]{65N30, 65N25, 65N50, 76D07.}
\begin{abstract}
In this paper we discuss the optimal convergence of a standard adaptive scheme
based on mixed finite element approximation to the solution of the eigenvalue
problem associated with the Stokes equations. 
The proofs of the quasi-orthogonality and the discrete reliability are presented. 
Our numerical experiments confirm the efficacy of the proposed adaptive scheme.
\end{abstract}
\maketitle


\section{Introduction}
Over the last few decades, the mathematical understanding of adaptive finite
element methods (AFEM) has reached its maturity and adaptive schemes are more
and more popular for the approximation of the solutions to partial
differential equations (PDEs). Their success is based on the possibility of
computing the approximated solutions of PDEs accurately with minimum number of
degree of freedoms (DOFs).
A rich literature confirms the huge impact of these methods in real life
applications.
We recall the milestone
works~\cite{ainsworth1997posteriori,verfurth2013posteriori} for the
development and analysis of a posteriori estimators which are key ingredients of
adaptive algorithms. 

In fluid mechanics, eigenvalue problems play a crucial role in understanding
the behavior of fluid flows and their stability.
Analyzing the eigenvalues provides valuable insights into the critical
conditions for stability or instability in fluid flows, aiding in the design
and prediction of fluid systems in various engineering applications.
In this context, the numerical study of the Stokes eigenvalue problem is of
great importance. 
Adaptive methods are designed in order to dynamically adjusts the
computational resources, based on the evolving solution characteristics.
Unlike static methods, adaptive techniques refine their discretization and
mesh adaptively during the computation, focusing computational efforts in
regions where the solution exhibits significant variations. This adaptability
allows for a more efficient and accurate determination of the solutions,
particularly when dealing with complex geometries or varying physical
parameters. By intelligently allocating computational resources, adaptive
methods enhance both the precision and the computational efficiency.

The a priori analysis of the Stokes problem is a classical topic within the
theory of mixed finite elements~\cite{bbf}. On the other hand, the study of
AFEM for the Stokes problem has been an open problem since the recent work
presented in~\cite{feischl2019optimality}.

In the literature, few results based on a posteriori error estimate for the
Stokes eigenvalue problem are available. In~\cite{lovadina2009posteriori}, it
is considered a residual based a posteriori error estimator for the Stokes
eigenvalue problem using stable approximations.
In~\cite{liu2013superconvergence}, it is presented a recovery type a
posteriori error estimator for the Stokes eigenvalue problems based on
projection method using some superconvergence results.
Also~\cite{han2015new, huang2015lower} deal with residual type a posteriori
error estimators. In~\cite{gedicke2018arnold}, a priori and a posteriori error
estimates are studied, for the stress velocity formulation of the Stokes
eigenvalue problem using Arnold--Winther mixed finite elements.  A posteriori error
estimates for divergence-conforming discontinuous Galerkin finite elements of
the Stokes eigenvalue problem is discussed in~\cite{gedicke2020divergence}. 
 
The results of the present paper extend to eigenvalue problems the seminal
contribution of Feischl who recently proved the optimality of a standard
adaptive finite element method for the Stokes problem
in~\cite{feischl2019optimality}. 
One of the main contributions of~\cite{feischl2019optimality} is the proof of
the quasi-orthogonality property. Specifically, it is established a relation
between the quasi-orthogonality property and the $LU$-factorization of
infinite matrices. Some  optimality results for adaptive mesh refinement
algorithms for non-symmetric, indefinite, and time-dependent problems are
discussed in~\cite{feischl2022inf}. More specifically, Feischl overcame one
technicality in the proof of the quasi-orthogonality
of~\cite{feischl2019optimality} and presented the proof of optimality.

{As far as the authors are aware, there is currently no optimality result available in 
the existing literature for a standard adaptive finite element method applied to the Stokes eigenvalue problems.} 
In this paper, we introduce
a suitable estimator for the Stokes eigenvalue problem which is locally
equivalent to the a posteriori error estimator proposed
in~\cite{lovadina2009posteriori} and we prove the optimality of the standard
adaptive finite element method stemming from this estimator.
Specifically, we establish the four key properties discussed in
\cite{carstensen2014axioms} to ensure the  rate optimality of adaptive finite
element methods. In particular, we prove the quasi-orthogonality property for
the numerical approximation of Stokes eigenvalue problem using the
$LU$-factorization of matrices discussed in~\cite{feischl2022inf}.

The rest of the paper is organised as follows:
Section \ref{sec1} is devoted to the function spaces and problem formulation. 
Error estimator and adaptive method will be presented in  Section \ref{sec3}. 
Proof of convergence and rate optimality are discussed in Section \ref{se:proofs}.
Finally, computational results are demonstrated in the Section \ref{sec5}.
\section{The Stokes Eigenvalue problem}\label{sec1}
Let $\Omega$ be an open, bounded, and connected domain with Lipschitz boundary
$\partial\Omega$. For simplicity, we assume that $\Omega$ is a polygon in two
dimensions or a polyhedron in three dimensions.
Let $H^s(\Omega)$ be the standard Sobolev space for $s\ge 0$. We denote the
scalar product in $L^2(\Omega)=H^0(\Omega)$ as $(\cdot,\cdot)$. If $\omega$ is
a subset of $\Omega$ then $(\cdot,\cdot)_\omega$ is the restriction of the
scalar product to $\omega$.
The following function spaces will be used to write our weak formulation:
\begin{align*}
\bV&=\bH^1_0(\Omega)=\{\bu\in \bH^{1}(\Omega):\bu|_{\partial\Omega}=\textbf{0}\}\\
Q&=L^2_0(\Omega)=\left\{p\in L^2(\Omega):\int_{\Omega} p=0\right\}.
\end{align*}
We consider the velocity-pressure Stokes eigenvalue problem: Find an eigenpair $(\bu,p,\lambda)$ with $\bu\neq \textbf{0}$ such that
\begin{subequations}\label{stokeseig}
\begin{align}
-\Delta \bu+\nabla p &=\lambda \bu\quad\;\mbox{in} \;\Omega\\
\vdiv \bu&=0 \qquad\mbox{in} \;\Omega\\
\bu&=0\qquad\mbox{on} \;\partial\Omega.
\end{align}
\end{subequations}
The standard weak formulation of the Stokes eigenvalue
problem~\eqref{stokeseig} is obtained as follows: Find $(\bu,p,\lambda)\in \bV\times
Q\times \mathbb{R}^{+}$ with $||\bu||_{0}=1$ such that
\begin{equation}\label{conweakstokeseig}
\aligned
&a(\bu,\bv)+b(\bv,p)=\lambda(\bu,\bv)&&\forall\bv\in\bV\\
&b(\bu,q)=0&&\forall q\in Q,
\endaligned
\end{equation}
where $a(\bu,\bv)=(\nabla \bu,\nabla\bv)$ and $b(\bv,p)=-(\rm{div} \,\bv,p)$.
Throughout the paper, we will use the following graph norm:
\begin{align*}
|||(\bu,p)|||^2:= ||\nabla \bu||_0^2+||p||_0^2.
\end{align*}

\subsection{Discrete formulation of the Stokes eigenvalue problem}\label{sec2}
Let $\cT$ be a shape-regular triangulation of $\Omega$. We denote by $T\in\cT$
one element of the triangulation, by $h_T$ its diameter and by $h_\cT$
the meshsize of $\cT$ (maximum element diameter). We will also need the set
$\mathcal{E}$ of all edges of $\cT$ and their length $h_e$. When no confusion
arises, $h$ will be used in place of $h_\cT$.

We use the standard
generalized Taylor--Hood finite element scheme~\cite{bbf} that is defined as
$\bV_\cT^p\times Q_{\cT}^{p-1}$ with $p\ge 2$, where
\begin{align*}
P^{p}(\cT)&:=\{v\in H^1(\Omega): v|_T \mbox{ is a polynomial with degree} \le
p,\ T\in\cT\}\\
\bV_\cT^p&:= (P^{p}(\cT))^2\cap (H^1_0(\Omega))^2,\\
Q_{\cT}^{p-1}&:= P^{p-1}(\cT)\cap L^2_0(\Omega).
\end{align*}
The discrete weak formulation of the Stokes eigenvalue problem is defined as
follows: Find $(\bu_\cT,p_\cT,\lambda_\cT)\in \bV_{\cT}^p\times
Q_{\cT}^{p-1}\times\mathbb{R}^{+}$ with $||\bu_\cT||_{0}=1$ such that
\begin{equation}\label{disform11}
\aligned
&a(\bu_{\cT},\bv)+b(\bv,p_\cT)=\lambda_{\cT}(\bu_\cT,\bv)&&\forall\bv\in\bV_\cT^p\\
&b(\bu_\cT,q)=0&&\forall q\in Q_{\cT}^{p-1}.
\endaligned
\end{equation}
\begin{remark}
{In} the abstract theory of eigenvalue problems in mixed form, the problem is
usually written in terms of an eigenspace corresponding only to the velocity
space, see~\cite{bbg2,bbf}. This is convenient in order to analyze schemes
that do not satisfy the inf-sup condition, thus allowing for a non-uniqueness
of the pressure space. In our case, we are choosing an inf-sup stable scheme
so that we decided to use a formulation of the problem where the eigenspace
is associated with both the velocity and the pressure.
\end{remark}

It is well-known that the Taylor--Hood element is inf-sup stable {with
the mild assumption on the mesh discussed in~\cite{bbf}}
and that its generalization to higher degrees is stable as
well~\cite{boffi1994stability,boffi1997three, bbf}{.} %
{Specifically, in two dimensions we assume that the mesh $\cT$ contains
at least three elements and in three dimensions that each element $T\in\cT$
has three non coplanar edges not lying on the boundary $\partial \Omega$
(which is true, for instance, if each element has an internal vertex).}
In the sequel, we make the notation shorter by introducing the space
\[
\cX_{\cT}^p=\bV_{\cT}^p\times Q_{\cT}^{p-1}
\]
and the corresponding notation $u_{\cT}=(\bu_{\cT},p_{\cT})\in\cX_{\cT}^p$
together with
\[
|||u_{\cT}|||=|||(\bu_{\cT},p_{\cT})|||.
\]
Given two subspaces $A$ and $B$ of the space $\cX_\cT^p$ we shall use the gap
defined as
\[
\delta(A,B)=\sup_{u\in A}\inf_{v\in B}|||u-v|||.
\]

The stability of the numerical scheme implies then the following result.
\begin{lemma}\label{stab11}
For any $u_{\cT}\in\cX_{\cT}^p$ there exists $v_{\cT}\in\cX^p_\cT$ with
$|||v_{\cT}|||\le C_1 |||u_{\cT}|||$
such that
\begin{align}\label{eq24}
C_2|||u_{\cT}|||^2\le B(u_{\cT},v_{\cT})
\end{align}
where $ B(u_{\cT},v_\cT)=a(\bu_{\cT},\bv_{\cT}){+}b(\bv_{\cT},p_\cT){+}b(\bu_\cT,q_{\cT})$ and, $C_1$ and $C_2$ are positive constants independent of $h_{\cT}$.
\end{lemma}
\begin{proof}
Employing the inf-sup condition in conjunction with the definition of the bilinear form $B(\cdot,\cdot)$ yields the stated outcome.
\end{proof}
\begin{lemma}There holds:
\begin{align}\label{cont11}
 B(u_{\cT},v_{\cT})\le C_b|||u_{\cT}|||\,|||v_{\cT}|||\quad \forall u_{\cT},v_{\cT}\in \cX_{\cT}^p.
 \end{align}
\end{lemma}
\begin{proof}
An application of Cauchy-Schwarz implies the stated result.
\end{proof}
\begin{lemma}
There holds:
\begin{align}\label{l2h1ineq}
||\bu-\bu_\cT||_{0}\le \rho(h_\cT) |||u-u_{\cT}|||,
\end{align}
where $\rho(h_{\cT})$ tends to zero as $h_{\cT}$ tends to zero.
\end{lemma}
\begin{proof}
Applying the $L^{2}$ error estimate derived in {\cite[Eq. 4.2]{verfurth1984error}} in conjunction with eigenvalue findings \todo{(elliptic eigenvalue problem)} from {\cite[Lemma 3.4]{babuvska1989finite}} and \cite[\todo{Section 13}]{boffi2010finite} results in the presented conclusion.
\end{proof}
\begin{theorem}[Identity I]\label{algide}
Let $(\bu,p,\lambda)\in \bV\times Q\times \mathbb{R}^{+}$ and
$(\bu_\cT,p_\cT,\lambda_\cT)\in \bV_\cT\times Q_\cT\times \mathbb{R}^{+}$ be
solutions of the weak formulations~\eqref{conweakstokeseig}
and~\eqref{disform11}, respectively. Then the following identity holds: 
\[
\aligned
&\lambda_\cT -\lambda
= a(\bu-\bu_{\cT},\bu-\bu_{\cT}){+}\todo{2}b(\bu-\bu_{\cT},p-p_{\cT})-\lambda||\bu-\bu_{\cT}||_0^2.
\endaligned
\]
\end{theorem}
\begin{proof}
Note that 
\[
\aligned
&a(\bu-\bu_{\cT},\bu-\bu_{\cT}){+}\todo{2}b(\bu-\bu_{\cT},p-p_{\cT}) \\
&\quad=a(\bu,\bu){+}\todo{2}b(\bu,p)+a(\bu_{\cT},\bu_{\cT}){+}\todo{2}b(\bu_{\cT},p_{\cT})
-2(a(\bu,\bu_{\cT}){+}b(\bu,p_{\cT}){+}b(\bu_{\cT},p))\\
&\quad=\lambda(\bu,\bu)+\lambda_\cT(\bu_\cT,\bu_\cT)-2\lambda(\bu,\bu_\cT).
\endaligned
\]
Using 
\begin{align*}
\lambda(\bu-\bu_\cT,\bu-\bu_\cT)=\lambda(\bu,\bu)+\lambda(\bu_\cT,\bu_\cT)-2\lambda(\bu,\bu_\cT)
\end{align*}
gives
\[
\aligned
&a(\bu-\bu_{\cT},\bu-\bu_{\cT}){+}\todo{2}b(\bu-\bu_{\cT},p-p_{\cT})\\
&\quad=-\lambda(\bu_\cT,\bu_\cT)+\lambda_\cT(\bu_\cT,\bu_\cT)+\lambda(\bu-\bu_\cT,\bu-\bu_\cT).
\endaligned
\]
This completes the proof.
\end{proof}
Similarly, we can prove the following identity.
 \begin{theorem}[Identity II]\label{algideii}
Let $\cTh\in\TT$ be any refinement of a given triangulation $\cT\in\TT$. Let $(\bu_{\cTh},p_{\cTh},\lambda_{\cTh})\in \bV_{\cTh}\times Q_{\cTh}\times
\mathbb{R}^{+}$ and $(\bu_\cT,p_\cT,\lambda_\cT)\in \bV_\cT\times Q_\cT\times
\mathbb{R}^{+}$ be solutions of the discrete weak formulation~\eqref{disform11}.
Then the following identity holds: 
\begin{align*}
\lambda_\cT-\lambda_{\cTh}&= a(\bu_{\cTh}-\bu_{\cT},\bu_{\cTh}-\bu_{\cT}){+}\todo{2}\b(\bu_{\cTh}-\bu_{\cT},p_{\cTh}-p_{\cT})
-\lambda_{\cTh}||\bu_{\cTh}-\bu_{\cT}||_0^2.
\end{align*}
\end{theorem}
\section{Error estimator and adaptive scheme}\label{sec3}
In this section we introduce the a posteriori error estimator that we are
using in our adaptive scheme and we state our main convergence and optimality
theorem.
We are also recalling the main tools needed for the proof in the spirit
of~\cite{dorfler1996convergent,dai2008convergence,boffi2017optimal}. The
actual proof of all the auxiliary results is postponed to the next section for
the sake of readability.

In \cite{gantumur2014convergence,feischl2019optimality,feischl2022inf}, a
residual based a posteriori error estimator for the Stokes source problem, is
discussed. That estimator is locally equivalent to the a posteriori error
estimator proposed by Verf\"urth in~\cite{verfurth2013posteriori}.
The local equivalence between the two estimators follows directly from the following estimates:
\begin{align}\label{est-11}
||{\vdiv \bu_{\cT}}||^2_{L^2(\Omega)}\le C_{11} \sum_{e\in\mathcal{E}}h_e||[\partial_n\bu_{\cT}]||^2_{L^2(e)}
\end{align}
and 
\begin{align}\label{est-12}
||\vdiv \bu_{\cT}|_T||^2_{L^2(\partial T)}\le C_{12}{h_T^{-1}}||{\vdiv \bu_{\cT}}||^2_{L^2(T)},
\end{align}
where $C_{11}$ and $C_{12}$ are positive constants. For more details, see also \cite{gantumur2014convergence,morin2008basic,bansch2002adaptive}.

We introduce a new estimator for the Stokes eigenvalue problem which is
locally equivalent to the one proposed in~\cite{lovadina2009posteriori}.
Specifically, our adaptive scheme is based on the following local error
estimator for all $T\in\cT$:
\begin{equation}
\aligned
\eta_T^2&:= h_T^2||\lambda_{\cT}\bu_{\cT}+\Delta \bu_{\cT}-\nabla p_{\cT}||_{L^2(T)}^2\\
&\quad+h_T||[\partial_n\bu_{\cT}]||^2_{L^2(\partial{T}\cap\Omega)}
+h_T||\vdiv \bu_{\cT}|_T||^2_{L^2(\partial T)}.
\endaligned
\label{eq:estimator}
\end{equation}
Given a triangulation $\cT$, our global estimator reads
\begin{align}
\eta(\cT):= \left(\sum_{T\in\cT}\eta_T^2\right)^{1/2}.
\end{align}
For a set of elements $\mathcal{M}$, we define the notation $\eta(\mathcal{M})$ as
\begin{align}
\eta(\mathcal{M}):= \left(\sum_{T\in\mathcal{M}}\eta_T^2\right)^{1/2}.
\end{align}

In order to keep our notation lighter, we consider a simple eigenvalue (that
is, an eigenvalue with multiplicity $1$). This fact has been already used
implicitly in the definition of the estimator. More general situations could
be considered in the spirit of~\cite{boffi2017optimal}.

Let $\lambda$ be a simple eigenvalue with the associated one dimensional
eigenspace $W={\rm \textbf{span}}(u)$ with $u=(\bu,p)$. Let $\lambda_\cT$ be
the approximated (discrete) eigenvalue with the associated one dimensional
eigenspace $W_\cT={\rm \textbf{span}}(u_\cT)$ with $u_\cT=(\bu_\cT,p_\cT)$. We
are interested in estimating the error in the eigenvalue approximation
$|\lambda-\lambda_\cT{|}$ and the gap between the continuous eigenspace $W$ and
the approximated eigenspace $W_\cT$ $\delta(W,W_\cT)$.

Next, we state the reliability and the efficiency results for our a posteriori estimator.
\begin{proposition}\label{releff11}
Let $(u,\lambda)=(\bu,p,\lambda)$ and
$(u_\cT,\lambda_\cT)=(\bu_\cT,p_\cT,\lambda_\cT)$ be the continuous
solution of~\eqref{conweakstokeseig} and the discrete solution
of~\eqref{disform11}, respectively. Then, the following reliability estimate
holds:
\begin{align}
|||u-u_\cT|||&\le C (\eta(\cT)+|\lambda-\lambda_\cT|+\lambda||\bu-\bu_\cT||_0),\\
|\lambda-\lambda_\cT|&\le C
(\eta(\cT)^2+|\lambda-\lambda_\cT|^2+\lambda^2||\bu-\bu_\cT||_0^2),
\end{align}
where $C$ is a positive constant independent of the mesh size $h_\cT$.
Moreover, the following local efficiency result holds:
\begin{align}
\eta_T \le
C\left(|||u-u_\cT|||_{\omega(T)}+\sum_{T'\subset\omega(T)}h_{T'}^{1/2}(|\lambda-\lambda_\cT|+\lambda||\bu-\bu_\cT||_{0,T'})\right),
\end{align}
where $\omega(T)$ is the union of all elements that share at least one edge
$(d=2)$ or one face $(d=3)$ with T. 
\end{proposition}
\begin{proof}
Note that our estimator for the Stokes eigenvalue problem  is locally
equivalent to the a posteriori error estimator proposed
in~\cite{lovadina2009posteriori} with (\ref{est-11}) and (\ref{est-12}).
Hence, the stated results directly follows from the reliability and the
efficiency results discussed in~\cite{lovadina2009posteriori}.
\end{proof}
\begin{corollary}
The following reliability estimates hold:
\begin{align}
|||u-u_\cT|||&\le C \eta(\cT)+\rho_{rel1}(h_\cT)|||u-u_\cT|||,\\
|\lambda-\lambda_\cT|&\le C \eta(\cT)^2+\rho_{rel2}(h_\cT)|||u-u_\cT|||^2,
\end{align} 
where $\rho_{rel1}(h{)}$ and $\rho_{rel2}(h)$ tend to zero as $h$ tends to zero. Moreover, it also holds:
\begin{align}
\eta_T \le C|||u-u_\cT|||_{\omega(T)},
\end{align}
{for sufficiently small $h_\cT$.}
\end{corollary}
\begin{proof}
Combining Proposition \ref{releff11} and  (\ref{l2h1ineq}) leads {to} the stated result. {For further details on the local efficiency bound, we refer to \cite[Page 1259]{boffi2019posteriori}.}
\end{proof}

We consider a standard adaptive scheme based on the error
estimator~\eqref{eq:estimator} and on the usual paradigm
SOLVE--ESTIMATE--MARK--REFINE. We describe it in \textbf{Algorithm 1}, where we
use the short notation $\bullet_\ell$ for $\bullet_{\cT_\ell}$. In practice,
\textbf{Algorithm 1} should be combined with a suitable stopping criterion
associated to a given error tolerance.

\begin{algorithm}
\textbf{Algorithm 1}: Adaptive FEM for the Stokes eigenvalue problem \\
\textbf{Input} Given initial mesh $\cT_0$ and parameter $0<\theta<1$.\\
For $\ell=0,1,\cdots$ do:
\begin{itemize}
\item \textbf{Solve}: Compute $(\bu_\ell,p_\ell,\lambda_\ell)$ from (\ref{disform11})
\item \textbf{Estimate}: Compute $\eta_T$ for all $T\in\cT_\ell$
\item \textbf{Mark}: Obtain a set $\mathcal{M}_\ell\subseteq \cT_\ell$ with minimum cardinality such that
\begin{align*}
\eta(\mathcal{M}_\ell)^2\ge \theta \eta(\cT_\ell)^2
\end{align*}
\item \textbf{Refine}: Apply newest-vertex-bisection to refine $\cT_\ell$ and to obtain a new mesh $\cT_{\ell+1}$
\end{itemize}
\textbf{Output}: Sequence of meshes $\cT_\ell$ and corresponding approximations of the eigenvalue $\lambda_\ell$ and the eigenfunction $(\bu_\ell,p_\ell)$.
\end{algorithm}

We are now in a position to introduce and state the main result of our paper.
The convergence of the adaptive scheme is usually assessed by considering
nonlinear approximation classes, see~\cite{bdd}. In particular, we are going
to use the following semi-norm
\[
|W|_{\mathcal{A}_s}=
\sup_{m\in\NN}m^s\inf_{\substack{\cT\in\TT\\ card(\cT)=m}}\delta(W,\cX^p_\cT),
\]
where $\TT$ is the set of admissible triangulations and the cardinality
$card(\cT)$ is the number of elements of the mesh $\cT$. {For more details about the semi-norm, we refer to~\cite{gallistl2015optimal,boffi2017optimal,bdd}.}

\begin{theorem}\label{optthm}
Let $\cT_\ell$ be the sequence of meshes generated by \textbf{Algorithm 1}
with {the sufficiently small} bulk parameter $\theta\in(0,1)$ starting from a given mesh $\cT_0$ {which satisfies the assumption discussed in Section \ref{sec2}}.
If the continuous eigenspace $W$ has bounded $\mathcal{A}_s$-seminorm for some
$s$, i.e. $|W|_{\mathcal{A}_s}<\infty$, then the output of the adaptive scheme
satisifies  the following optimal estimate:
\begin{align}
\delta(W,W_{\ell})\le C(card(\cT_\ell)-card(\cT_0))^{-s}|W|_{\mathcal{A}}^s.
\end{align}
Moreover, it also holds:
\begin{align}
|\lambda-\lambda_\ell|\le C\,\delta(W,W_{\ell})^2.
\end{align}
\end{theorem}

The proof of the theorem together with the proof of a series of preparatory
results, will be given in Section~\ref{se:proofs}.

\section{Proof of convergence and rate optimality}
\label{se:proofs}
We follow the abstract setting used in~\cite{boffi2019adaptive}, see
also~\cite{carstensen2014axioms}.
Hence, the result of Theorem~\ref{optthm} follows from a series of preparatory
results that we are going to state and prove in this section.

\begin{proposition}[\textbf{Stability on nonrefined elements}]
Let $\cTh\in\TT$ be any refinement of a given triangulation $\cT\in\TT$ and
$S\subseteq\cT\cap\cTh$ be the set of the nonrefined elements. Then
\begin{equation}
\left|\left(\sum_{T\in S}\eta_{T}(\cTh)^2\right)^{1/2}-
\left(\sum_{T\in S}\eta_{T}(\cT)^2\right)^{1/2}\right|\le C_{stab}
|||u_{\cTh}-u_\cT|||.
\label{eq:stabnonref}
\end{equation}
\end{proposition}

\begin{proof}
Let $\cTh\in\TT$ be any refinement of the given triangulation $\cT\in\TT$ and
$S\subseteq\cT\cap\cTh$ be the set of the nonrefined elements. Then
\[
\aligned
&\left(\sum_{T\in S}\eta_{T}(\cTh)^2\right)^{1/2}\le\left(\sum_{T\in S}\eta_{T}(\cT)^2\right)^{1/2}\\
&\quad+\Bigg(\sum_{T\in S}\big(h_T^2||\lambda_{\cTh}\bu_{\cTh}+\Delta \bu_{\cTh}-\nabla p_{\cTh}-(\lambda_{\cT}\bu_{\cT}+\Delta \bu_{\cT}-\nabla p_{\cT})||_{L^2(T)}^2\\
&\qquad\quad+h_T||[\partial_n(\bu_{\cTh}-\bu_{\cT})]||^2_{L^2(\partial{T}\cap\Omega)}+h_T||\vdiv
(\bu_{\cTh}-\bu_{\cT})|_T||^2_{L^2(\partial T)}\big)\Bigg)^{1/2}
\endaligned
\]
and 
\[
\aligned
&\left(\sum_{T\in S}\eta_{T}(\cT)^2\right)^{1/2}\le\left(\sum_{T\in S}\eta_{T}(\cTh)^2\right)^{1/2}\\
&\quad+\Bigg(\sum_{T\in S}\big(h_T^2||\lambda_{\cTh}\bu_{\cTh}+\Delta \bu_{\cTh}-\nabla p_{\cTh}-(\lambda_{\cT}\bu_{\cT}+\Delta \bu_{\cT}-\nabla p_{\cT})||_{L^2(T)}^2\\
&\qquad\quad+h_T||[\partial_n(\bu_{\cTh}-\bu_{\cT})]||^2_{L^2(\partial{T}\cap\Omega)}+h_T||\vdiv
(\bu_{\cTh}-\bu_{\cT})|_T||^2_{L^2(\partial T)}\big)\Bigg)^{1/2}.
\endaligned
\]
By combining the above two estimates, we derive the stated result.
\end{proof}

\begin{proposition}[\textbf{Reduction property on refined elements}]
For any refinement $\cTh\in\TT$ of the given triangulation $\cT\in\TT$, the following condition holds:
\begin{equation}
\sum_{T\in\cTh\setminus\cT}\eta_{T}(\cTh)^2\le q_{red}
\sum_{T\in\cT\setminus\cTh}\eta_{T}(\cT)^2+C_{red}|||u_{\cTh}-u_\cT|||^2
\label{eq:redprop}
\end{equation}
{with $q_{red} <1$.}
\end{proposition}

\begin{proof}
Let $\cTh\in\TT$ be any refinement of the given triangulation $\cT\in\TT$.
For any $\delta>0$, using standard arguments
from~\cite{cascon2008quasi,gantumur2014convergence,boffi2019adaptive} gives
\[
\sum_{T\in\cTh}\eta_{T}(\cTh)^2\le
(1+\delta)\sum_{T\in\cT}\eta_{T}(\cT)^2
-\tau(1+\delta)\sum_{T\in\cT\setminus\cTh}\eta_{T}(\cT)^2
+C_{\delta}|||u_{\cTh}-u_\cT|||^2,
\]
where $C_\delta$ is a positive constant which may depend on $\delta$ and
$\tau$ is a positive constant which does not depend of $\delta$.
Let $\theta\in(0,1]$ be the marking parameter used in the following condition
\[
\theta\sum_{T\in\cT}\eta_{T}(\cT)^2\le\sum_{T\in\cT\setminus\cTh}\eta_{T}(\cT)^2.
\]
Hence, we obtain
\[
\sum_{T\in \cTh}\eta_{T}(\cTh)^2\le (1+\delta)(1-\tau\theta)\sum_{T\in \cT}\eta_{T}(\cT)^2+C_{\delta} |||u_{\cTh}-u_\cT|||^2.
\]
Choosing the parameter $\delta$ small enough, there exists a $\Theta$  with
$\Theta<1$ such that
\[
\sum_{T\in \cTh}\eta_{T}(\cTh)^2\le \Theta\sum_{T\in \cT}\eta_{T}(\cT)^2+C_{\delta} |||(\bu_{\cTh}-\bu_\cT,p_{\cTh}-p_\cT)|||^2.
\]
Combining the estimate with the stability on non-refined elements leads to
the desired result.
\end{proof}

\begin{proposition}[\textbf{Generalized quasi-orthogonality}]\label{prop:quasi}
The output of \textbf{Algorithm~1}
satisfies, for all $\ell$ and $N\in\NN$,
\begin{equation}\label{gqo}
\sum_{k=0}^{N}||u_{\ell+k+1}-u_{\ell+k}||^2_{\mathcal{X}} \le \widetilde{C(N)} ||u-u_\ell||^2_{\mathcal{X}}+\rho_{3}(h_{\ell})\sum_{k=0}^N ||u-u_{\ell+k}||^2_{\mathcal{X}},
\end{equation}
where $\rho_3(h)$ tends to zero as $h$ goes to zero and {$\widetilde{C(N)}=o(N)$ as $N\rightarrow \infty$. }
\end{proposition}

\begin{proof}
{
By leveraging the well-posedness of the continuous source problem
 \begin{align}\label{sourcepre12}
 \begin{split}
 -\Delta \bm{u}^{\bm{f}}+\nabla p^{\bm{f}} &= \bm{f}\quad \mbox{in}\;\Omega,  \\
\nabla\cdot \bm{u}^{\bm{f}}&=0\quad \;\;\;\mbox{in}\;\Omega,\\
 \bm{u}^{\bm{f}}&=\bm{0}\quad \;\;\;\mbox{on}\;\partial\Omega,
 \end{split}
\end{align}  
with the compatibility relation
\begin{align*}
\int_{\Omega} p^{\bm{f}}\;dx=0,
\end{align*}
 the operators $T:\bm{L}^2(\Omega)\rightarrow\bm{H}^1_0(\Omega)$ and $S:\bm{L}^2(\Omega)\rightarrow L^2_0(\Omega)$
are well defined for any $\bm{f}\in\bm{L}^2(\Omega)$, where
$T\bm{f}=\bm{u}^{\bm{f}}$ and $S\bm{f}=p^{\bm{f}}$ 
represent the velocity and pressure components, respectively, of the solution
to problem \eqref{sourcepre12}.}
\par
{For given $\bm{f}\in \bm{L}^2(\Omega)$, $\textbf{T}\bm{f}:=(T\bm{f},
S\bm{f})\in \bm{H}_0^1(\Omega)\times L^2_0(\Omega)$ is the solution $(\bu^{\bm{f}},p^{\bm{f}})$ of
\begin{align}\label{nmsprob12source}
&a(\bu^{\bm{f}},\bv)+b(\bv,p^{\bm{f}})=(\bm{f},\bv)&&\forall\bv\in\bV\\
&b(\bu^{\bm{f}},q)=0&&\forall q\in Q.
\end{align}
Given that the discrete source problem is also well-posed, we similarly define the operators
$T_h:\bm{L}^2(\Omega)\rightarrow\bm{V}_h$ and $S_h:\bm{L}^2(\Omega)\rightarrow Q_h$, where
$T_{h}\bm{f}=\bm{u}_h^{\bm{f}}$ and $S_h\bm{f}=p_h^{\bm{f}}$ denote the discrete velocity and discrete pressure approximations, respectively.
For given $\bm{f}\in \bm{L}^2(\Omega)$, $\textbf{T}_h\bm{f}:=(T_h\bm{f},
S_h\bm{f})\in\bV_h\times Q_h$ is the solution $(\bu_h^{\bm{f}},p_h^{\bm{f}})$ of
\begin{align}\label{nmsprob12sourcedis}
&a(\bu_h^{\bm{f}},\bv)+b(\bv,p_h^{\bm{f}})=(\bm{f},\bv)&&\forall\bv\in\bV_h\\
&b(\bu_h^{\bm{f}},q)=0&&\forall q\in Q_h.
\end{align}
Let $\lambda$ denote the eigenvalue associated with the eigenfunction $\bm{u}$. Then, the following relationships hold:
\begin{align*}
\bm{u}_{\ell+k+1}-\bm{u}_{\ell+k}&=\lambda_{\ell+k+1} T_{\ell+k+1}\bm{u}_{\ell+k+1}-\lambda_{\ell+k} T_{\ell+k}\bm{u}_{\ell+k}\\
&=(\lambda_{\ell+k+1}-\lambda_{\ell+k})T_{\ell+k+1}\bm{u}_{\ell+k+1}+\lambda_{\ell+k}(T_{\ell+k+1}-T_{\ell+k})\bm{u}_{\ell+k+1}\\
&\quad+\lambda_{\ell+k} T_{\ell+k}(\bm{u}_{\ell+k+1}-\bm{u}_{\ell+k})\\
&=(\lambda_{\ell+k+1}-\lambda_{\ell+k})T_{\ell+k+1}\bm{u}_{\ell+k+1}+\lambda_{\ell+k}(T_{\ell+k+1}-T_{\ell+k})(\bm{u}_{\ell+k+1}-\bm{u})\\
&\quad+\lambda_{\ell+k}(T_{\ell+k+1}-T_{\ell+k})\bm{u}+\lambda_{\ell+k} T_{\ell+k}(\bm{u}_{\ell+k+1}-\bm{u}_{\ell+k})
\end{align*}
and 
\begin{align*}
p_{\ell+k+1}-p_{\ell+k}&=(\lambda_{\ell+k+1}-\lambda_{\ell+k})S_{\ell+k+1}\bm{u}_{\ell+k+1}+\lambda_{\ell+k}(S_{\ell+k+1}-S_{\ell+k})(\bm{u}_{\ell+k+1}-\bm{u})\\
&\quad+\lambda_{\ell+k}(S_{\ell+k+1}-S_{\ell+k})\bm{u}+\lambda_{\ell+k} S_{\ell+k}(\bm{u}_{\ell+k+1}-\bm{u}_{\ell+k}).
\end{align*}
Let $\textbf{T}u =(T\bm{u},S\bm{u})$ and $\textbf{T}_{\ell+k}u_{\ell+k} =(T_{\ell+k}\bm{u}_{\ell+k},S_{\ell+k}\bm{u}_{\ell+k})$. Then,
it follows that
\begin{align*}
||{u}_{\ell+k+1}-{u}_{\ell+k}&||_{\cX}^2\le\todo{3}(  |\lambda_{\ell+k+1}-\lambda_{\ell+k}|^2 ||\textbf{T}_{\ell+k+1}{u}_{\ell+k+1}||_{\cX}^2\\
&\quad+\lambda_{\ell+k}^2||{(\textbf{T}_{\ell+k+1}-\textbf{T}_{\ell+k})({u}-u_{\ell+k+1})}||_{\cX}^2\\
&\quad+\lambda_{\ell+k}^2||{(\textbf{T}_{\ell+k+1}-\textbf{T}_{\ell+k}){u}}||_{\cX}^2+\lambda_{\ell+k}^2||{\textbf{T}_{\ell+k}({u}_{\ell+k+1}-{u}_{\ell+k})}||_{\cX}^2).
\end{align*}
Hence, we obtain
\begin{align*}
\sum_{k=0}^{N}||{u}_{\ell+k+1}-{u}_{\ell+k}&||_{\cX}^2\le (I) + (II) + \todo{3}\sum_{k=0}^N \lambda_{\ell+k}^2||{(\textbf{T}_{\ell+k+1}-\textbf{T}_{\ell+k}){u}}||_{\cX}^2 + (III),
\end{align*}
\todo{where
\begin{align*}
(I)&:= \todo{3} \sum_{k=0}^N |\lambda_{\ell+k+1}-\lambda_{\ell+k}|^2 ||\textbf{T}_{\ell+k+1}{u}_{\ell+k+1}||_{\cX}^2,\\
(II)&:=\todo{3} \sum_{k=0}^N\lambda_{\ell+k}^2||{(\textbf{T}_{\ell+k+1}-\textbf{T}_{\ell+k})({u}-u_{\ell+k+1})}||_{\cX}^2,\\
(III)&:=\todo{3}\sum_{k=0}^N\lambda_{\ell+k}^2||{\textbf{T}_{\ell+k}({u}_{\ell+k+1}-{u}_{\ell+k})}||_{\cX}^2.
\end{align*}}
Using the quasi-orthogonality estimate from the source problem discussed in \cite[Eq. 8]{feischl2022inf} yields 
\begin{align*}
\todo{3}\sum_{k=0}^N \lambda_{\ell+k}^2&||{(\textbf{T}_{\ell+k+1}-\textbf{T}_{\ell+k}){u}}||_{\cX}^2 \\
\le &C(N) \bar{\lambda}_{\ell}^2 ||{(\textbf{T}-\textbf{T}_{\ell}){u}}||_{\cX}^2\\
\le & C(N)\bar{\lambda}_{\ell}^2 ((1+c)||{\textbf{T}u-\textbf{T}_{\ell}{u}_{\ell}}||_{\cX}^2+(1+1/c) ||{\textbf{T}_{\ell}(u-{u}_\ell})||_{\cX}^2),
\end{align*}
where $c\in (0,1]$, $\bar{\lambda}_{\ell}^2=\max_{0\le k \le N}\lambda_{\ell+k}^2$ and ${C(N)}=o(N)$ as $N\rightarrow \infty$.\\
Using the discrete inf-sup (\ref{eq24}) and the definition of the weak source problem with the force function $f=\bm{u}-\bm{u}_{\ell}$ and then applying the Cauchy-Schwartz with the bound from~\eqref{l2h1ineq} implies that
\begin{align}\label{disbound11}
C_2 ||{\textbf{T}_{\ell}(u-{u}_\ell})||_{\cX}^2\le B({\textbf{T}_{\ell}(u-{u}_\ell}),v)&=(\bu-\bu_{\ell}, \bm{v})\nonumber\\
 &\le \rho(h_{\ell}) ||{(u-{u}_\ell})||_{\cX} C_1 ||{\textbf{T}_{\ell}(u-{u}_\ell})||_{\cX}, 
\end{align}
where $\rho(h)$ tends to zero as $h$ goes to zero.\\
From the estimate above and the definitions  ${u}=\lambda T u$ and ${u}_\ell=\lambda_\ell T_\ell u_\ell$, together with  the bound from~\eqref{l2h1ineq}, it follows that
\begin{align*}
\todo{3}\sum_{k=0}^N \lambda_{\ell+k}^2&||{(\textbf{T}_{\ell+k+1}-\textbf{T}_{\ell+k}){u}}||_{\cX}^2 \\
\le & C(N) \bar{\lambda}_{\ell}^2((1+c)||{u/\lambda-{u}_{\ell}/\lambda_{\ell}}||_{\cX}^2+(1+1/c) \rho_2(h_\ell)||{u-{u}_\ell}||_{\cX}^2)\\
\le & C(N)((1+c)c_1(|\lambda-\lambda_\ell|^2/\lambda^2||u||_{\cX}^2+||{u-{u}_{\ell}}||_{\cX}^2)+(1+1/c) \bar{\lambda}_{\ell}^2\rho_2(h_\ell)||{u-{u}_\ell}||_{\cX}^2)\\
\le & C(N)((1+c)c_1(1+\rho_{2}(h_\ell)^2||\textbf{T}u||_{\cX}^2)||{u-{u}_{\ell}}||_{\cX}^2+(1+1/c) \bar{\lambda}_{\ell}^2\rho_2(h_\ell)||{u-{u}_\ell}||_{\cX}^2)\\
=& \widetilde{C(N)} ||{u-{u}_{\ell}}||_{\cX}^2,
\end{align*}
where
$\widetilde{C(N)}=C(N)((1+c)c_1(1+\rho_{2}(h_\ell)^2||\textbf{T}u||_{\cX}^2)+(1+1/c)
\bar{\lambda}_{\ell}^2\rho_2(h_\ell))$ with
$c_1=\bar{\lambda}_{\ell}^2/\lambda_\ell^2$ and where $\rho_2(h)$ tends to zero as $h$ goes to zero.
The identities provided in Theorems~\ref{algide} and~\ref{algideii} 
along with the bound from~\eqref{l2h1ineq} finally give:
\begin{align*}
(I)&\le\rho_{2}(h_{\ell})^2\sum_{k=0}^N ||u-u_{\ell+k}||^2_{\mathcal{X}}||\textbf{T}_{\ell+k+1}{u}_{\ell+k+1}||_{\cX}^2
\end{align*}
\todo{Using 
\[||{(\textbf{T}_{\ell+k+1}-\textbf{T}_{\ell+k})({u}-u_{\ell+k+1})}||_{\cX} \le ||{(\textbf{T}_{\ell+k+1}-\textbf{T}_{\ell+k})||_{\mathcal{L}(\cX,\cX)} ||({u}-u_{\ell+k+1})}||_{\cX}\]
implies the following bound}
 \begin{align*}
(II) &\le \bar{\lambda}_{\ell}^2\rho_{2}(h_{\ell})^2\sum_{k=0}^N ||u-u_{\ell+k}||^2_{\mathcal{X}}
\end{align*}
\todo{Analogous to the bound in (\ref{disbound11}), by employing the discrete inf-sup condition (\ref{eq24}), invoking the definition of the weak source problem with the forcing term $f=\bm{u}_{\ell+k+1}-\bm{u}_{\ell+k}$, and applying the Cauchy–Schwarz inequality together with the bound from~\eqref{l2h1ineq}, we derive the following estimate}
\begin{align*}
(III)&\le \bar{\lambda}_{\ell}^2\rho_{2}(h_{\ell})^2\sum_{k=0}^N ||u-u_{\ell+k}||^2_{\mathcal{X}}.
\end{align*}
Combining the above estimates leads to the stated result.}
\end{proof}
\begin{proposition}[\textbf{Discrete reliability}]
For all {refinements}~$\cTh\in\TT$ of a triangulation $\cT\in\TT$, there exists a
subset $\cR(\cT,\cTh)\subseteq\cT$ with
$\cT\setminus\cTh\subseteq\cR(\cT,\cTh)$ and
$|\cR(\cT,\cTh)|\le C_{ref}|\cT\setminus\cTh|$ such that
\begin{equation}
|||u_{\cTh}-u_{\cT}|||\le C\left(\sum_{T\in \cR(\cT,\cTh)}\eta_{T}^2\right)^{1/2}
+\rho_1(h_{\cT})(|||u-u_\cT|||+|||u-u_{\cTh}|||).
\label{eq:drel}
\end{equation}
\end{proposition}

\begin{proof}

For any $(\bv_{\cT},q_{\cT})\in\bV_{\cT}\times Q_\cT$ and
$(\bv_{\cTh},q_{\cTh})\in\bV_{\cTh}\times Q_{\cTh}$, the discrete weak
formulation gives 
\[
\aligned
&a(\bu_{\cTh}-\bu_{\cT},\bv_{\cTh}){+}b(\bv_{\cTh},p_{\cTh}-p_{\cT}){+}b(\bu_{\cTh}-\bu_{\cT},q_{\cTh})\\
&\quad=a(\bu_{\cTh}-\bu_{\cT},\bv_{\cTh}-\bv_{\cT}){+}b(\bv_{\cTh}-\bv_{\cT},p_{\cTh}-p_{\cT}){+}b(\bu_{\cTh}-\bu_{\cT},q_{\cTh}-q_{\cT})\\
&\qquad+(\lambda_{\cTh}\bu_{\cTh}-\lambda_{\cT}\bu_{\cT},v_{\cT})\\
&\quad=(\lambda_{\cTh}\bu_{\cTh},\bv_{\cTh}-\bv_{\cT})-(a(\bu_{\cT},\bv_{\cTh}-\bv_{\cT})+b(\bv_{\cTh}-\bv_{\cT},p_{\cT})+b(\bu_{\cT},q_{\cTh}-q_{\cT}))\\
&\qquad+(\lambda_{\cTh}\bu_{\cTh}-\lambda_{\cT}\bu_{\cT},\bv_{\cT})\\
&\quad=(\lambda_{\cT}\bu_{\cT},\bv_{\cTh}-\bv_{\cT})-(a(\bu_{\cT},\bv_{\cTh}-\bv_{\cT})+b(\bv_{\cTh}-\bv_{\cT},p_{\cT})+b(\bu_{\cT},q_{\cTh}-q_{\cT}))\\
&\qquad+(\lambda_{\cTh}\bu_{\cTh}-\lambda_{\cT}\bu_{\cT},\bv_{\cTh})+(\lambda_{\cTh}\bu_{\cTh}-\lambda_{\cT}\bu_{\cT},\bv_{\cTh}-\bv_{\cT}).
\endaligned
\]
An application of integration by parts implies that
\[
\aligned
&a(\bu_{\cTh}-\bu_{\cT},\bv_{\cTh}){+}b(\bv_{\cTh},p_{\cTh}-p_{\cT}){+}b(\bu_{\cTh}-\bu_{\cT},q_{\cTh})\\
&\quad=(\lambda_{\cT}\bu_{\cT},\bv_{\cTh}-\bv_{\cT})-(a(\bu_{\cT},\bv_{\cTh}-\bv_{\cT})+b(\bv_{\cTh}-\bv_{\cT},p_{\cT})+b(\bu_{\cT},q_{\cTh}-q_{\cT}))\\
&\qquad+(\lambda_{\cTh}\bu_{\cTh}-\lambda_{\cT}\bu_{\cT},\bv_{\cTh})+(\lambda_{\cTh}\bu_{\cTh}-\lambda_{\cT}\bu_{\cT},\bv_{\cTh}-\bv_{\cT})\\
&\quad=\sum_{T\in\cT}\int_T(\lambda_{\cT}\bu_{\cT}+\Delta\bu_{\cT}-\nabla p_{\cT})\cdot (\bv_{\cTh}-\bv_{\cT}) - \sum_{E\in\cT}\int_E [\partial_n\bu_{\cT}]\cdot(\bv_{\cTh}-\bv_{\cT})\\
&\qquad+\int_{\Omega}(q_{\cTh}-q_{\cT})\nabla\cdot \bu_{\cT}
+(\lambda_{\cTh}\bu_{\cTh}-\lambda_{\cT}\bu_{\cT},\bv_{\cT}).
\endaligned
\]
Next we use the notation $\omega$ for the interior of the region covered by
the refined triangles, i.e. $\omega=int\cup_{T\in \cT\setminus \cTh}\bar{T}$.
Now we set $(\bv_{\cT},q_{\cT})$ to be equal to $(\bv_{\cTh}, q_{\cTh})$ in
$\Omega\setminus \omega$ and equal to the Scott--Zhang interpolator of
$(\bv_{\cTh},q_{\cTh})$ in $\omega$.
Then, we have
\begin{align*}
a(&\bu_{\cTh}-\bu_{\cT},\bv_{\cTh}){+}b(\bv_{\cTh},p_{\cTh}-p_{\cT}){+}b(\bu_{\cTh}-\bu_{\cT},q_{\cTh})\\
&\le \sum_{T\in\cT\setminus\cTh}||\lambda_{\cT}\bu_{\cT}+\Delta\bu_{\cT}-\nabla p_{\cT}||_{L^2(T)} ||\bv_{\cTh}-\bv_{\cT}||_{L^2(T)} \\
&+ \sum_{E\in\cT\setminus \cTh}|| [\partial_n\bu_{\cT}]||_{L^2(E)}||\bv_{\cTh}-\bv_{\cT}||_{L^2(E)}\\
&+ \sum_{T\in\cT\setminus\cTh}||q_{\cTh}-q_{\cT}||_{L^2(T)} ||\nabla\cdot \bu_{\cT}||_{L^2(T)}\\ 
&+ \sum_{T\in\cT\setminus\cTh}||\lambda_{\cTh}\bu_{\cTh}-\lambda_{\cT}\bu_{\cT}||_{L^2(T)} ||\bv_{\cTh}-\bv_{\cT}||_{L^2(T)}\\
&\quad+
\sum_{T\in\cT}||\lambda_{\cTh}\bu_{\cTh}-\lambda_{\cT}\bu_{\cT}||_{L^2(T)}
||\bv_{\cTh}||_{L^2(T)}.
\end{align*}
By Theorem \ref{algide} and Theorem \ref{algideii}, we get the following identities:
\begin{align*}
\lambda_\cT -\lambda&= a(\bu-\bu_{\cT},\bu-\bu_{\cT}){+}b(\bu-\bu_{\cT},p-p_{\cT}){+}b(\bu-\bu_{\cT},p-p_{\cT})\\
&\quad-\lambda||\bu-\bu_{\cT}||_0^2,\\
\lambda_\cT-\lambda_{\cTh}&= a(\bu_{\cTh}-\bu_{\cT},\bu_{\cTh}-\bu_{\cT}){+}b(\bu_{\cTh}-\bu_{\cT},p_{\cTh}-p_{\cT}){+}b(\bu_{\cTh}-\bu_{\cT},p_{\cTh}-p_{\cT})\\
&\quad-\lambda_{\cTh}||\bu_{\cTh}-\bu_{\cT}||_0^2.
\end{align*}
Using the error estimates for $|||(\bu-\bu_\cT,p-p_\cT)|||$ and $||\bu-\bu_\cT||_{0}$ with (\ref{l2h1ineq}), it follows that 
\begin{align}\label{enq11}
|\lambda_\cT-\lambda_{\cTh}|+||\bu_{\cTh}-\bu_{\cT}||_0\le \rho_1(h_{\cT}) |||(\bu_{\cTh}-\bu_\cT,p_{\cTh}-p_\cT)|||,
\end{align}
where $\rho_1(h_{\cT})$ tends to zero as $h_{\cT}$ tends to zero.\\
Since $(\bu_{\cTh}-\bu_{\cT},p_{\cTh}-p_{\cT})\in \bV_{\cTh} \times Q_{\cTh}$,
using the stability estimate~\eqref{stab11} with~\eqref{enq11} leads to the
stated result.
\end{proof}

\begin{remark}\label{remark1}
By \cite{feischl2022inf}, we know that the block $LU$ factorization  of the matrix $A$ satisfies the following properties 
\begin{align*}
||U||_2\le C_1 N^{1/2-\delta},\quad ||U^{-1}(:,j)||_2\le C_2,
\end{align*}
where $C_1$ is a positive constant depending on $C_a$ and $\gamma$, and $C_2$
is a positive constant depending on $\gamma$. This implies 
\begin{align*}
C(N)\le C_3 N^{1-2\delta},
\end{align*}
where $C_3$ is the positive constant depending on $C_1$ and $C_2$.
\end{remark}

\subsection{Linear convergence}
In this section, we establish the linear convergence of our estimator by leveraging the quasi-orthogonality discussed in Proposition \ref{prop:quasi}. The attainment of rate-optimality relies on the pivotal role played by linear convergence in the proof.
\begin{lemma}\label{lem-lin11}
Let the proposed estimator $\eta_\ell$ satisfy the following reduction property and reliability bound
\begin{align}
\eta(\cT_{\ell+1})^2&\le \kappa \eta(\cT_{\ell})^2+C_{est} ||u_{{\ell+1}}-u_{{\ell}}||^2_{\mathcal{X}},\label{red1}\\
||u-u_{\ell}||_{\mathcal{X}}&\le C_{rel}\eta(\cT_{\ell}) + \rho_1(h_{\ell}) ||u-u_\ell||_{\mathcal{X}},\label{rel1}
\end{align} 
for all $\ell\in \mathbb{N}$ and some parameter $0<{\kappa<1}$ with $C_{est}>0$. If the quasi orthogonality holds then 
\begin{align}\label{quasilem1}
\sum_{k=l}^{k=l+{N}}(\eta(\cT_{k})^2+\beta||u-u_{k}||^2_{\mathcal{X}})\le D(N) \eta^2(\cT_\ell)+ \beta||u-u_{\ell}||^2_{\mathcal{X}}, \quad\forall \ell,N\in\mathbb{N}
\end{align}
where $D(N)= 1+(\widetilde{\kappa}+C_{est}(1+2\epsilon C_{rel}^2)C(N-1)2C_{rel}^2)/(1-\widetilde{\kappa})$, $ \beta=(\epsilon-{2\epsilon\rho_1(h_\ell)^2}-\rho_{2}(h_{\ell}))/(1-\widetilde{\kappa})$ {with $\widetilde{\kappa}=\kappa(1+2\epsilon C_{rel}^2)<1$ for small enough $\epsilon >0$} and $C(N)$ is given in (\ref{gqo}). Moreover, we have 
\begin{align}\label{quasilem11}
\sum_{k=l}^{k=l+{N}}\xi_{k}^2\le D(N) \xi^2(\cT_\ell), \quad\forall \ell,N\in\mathbb{N}
\end{align}
{where $\xi_{k}^2:=  \xi^2 (\cT_k)=\eta_{k}^2+\beta||u-u_{k}||^2_{\mathcal{X}}$ with $\eta_k=  \eta (\cT_k)$}.
\end{lemma}
\begin{proof}
Using the reduction property (\ref{red1}) implies  
\begin{align*}
\sum_{k=l+1}^{k=l+{N}}&(\eta(\cT_{k})^2+\epsilon||u-u_{k}||^2_{\mathcal{X}})\\
&\le (1+2\epsilon C_{rel}^2)\sum_{k=l+1}^{k=l+{N}}\eta(\cT_{k})^2+\sum_{k=l+1}^{k=l+{N}}{2\epsilon\rho_1(h_\ell)^2}||u-u_{k}||^2\\
&\le (1+2\epsilon C_{rel}^2)\Big(\kappa\sum_{k=l+1}^{k=l+{N}}\eta(\cT_{k-1})^2+C_{est} \sum_{k=l+1}^{k=l+{N}}||u_{{k}}-u_{{k-1}}||^2_{\mathcal{X}}\Big)\\
&\quad +\sum_{k=l+1}^{k=l+{N}}{2\epsilon\rho_1(h_\ell)^2}||u-u_{k}||^2{.}
\end{align*}
By the quasi orthogonality (\ref{gqo}), the following bound holds:
\begin{align*}
\sum_{k=l+1}^{k=l+{N}}&(\eta(\cT_{k})^2+\epsilon||u-u_{k}||^2_{\mathcal{X}})\\
&\le \widetilde{\kappa}\sum_{k=l+1}^{k=l+{N}}\eta(\cT_{k-1})^2+C_{est}(1+2\epsilon C_{rel}^2)C(N-1)||u-u_{{\ell}}||^2_{\mathcal{X}}\\
&\quad+ \rho_{2}(h_{\ell})\sum_{k=l+1}^{k=l+{N}}||u-u_{k-1}||^2_{\mathcal{X}}+ {2\epsilon\rho_1(h_\ell)^2}\sum_{k=l+1}^{k=l+{N}}||u-u_{k}||^2_{\mathcal{X}},
\end{align*}
where $\widetilde{\kappa}=\kappa(1+2\epsilon C_{rel}^2)<1$ for small enough $\epsilon >0$.
Applying the reliability bound (\ref{rel1}) gives
\begin{align*}
(1-\widetilde{\kappa})\sum_{k=l+1}^{k=l+{N}}&\eta(\cT_{k})^2+(\epsilon-{2\epsilon\rho_1(h_\ell)^2}-\rho_{2}(h_{\ell}))\sum_{k=l+1}^{k=l+{N}}||u-u_{k}||^2_{\mathcal{X}}\\
&\le (\widetilde{\kappa} +C_{est}(1+2\epsilon C_{rel}^2)C(N-1)2C_{rel}^2) \eta(\cT_{\ell})^2
+ \rho_{2}(h_\ell)||u-u_{\ell}||^2_{\mathcal{X}}.
\end{align*}
This completes the proof.
\end{proof}
\begin{lemma}\label{lem-lin12}
Let the sequence $\{\eta(\cT_\ell)\}_{\ell\in\mathbb{N}}\subset \mathbb{R}$ satisfies (\ref{quasilem1}). Then there exists ${N}_0\in\mathbb{N}$ such that
\begin{align*}
q_{\log}:=\log (D(N_0))-\sum_{j=1}^{N_0}D(j)^{-1}<0.
\end{align*}
If the estimator $\{\xi_\ell\}_{\ell\in\mathbb{N}}$ satisfies the following monotone condition 
\begin{align*}
\xi_{\ell+k}^2\le C_{mon} \xi^2_\ell,
\end{align*}
where 
\begin{align*}
 \xi^2_\ell =\eta^2(\cT_\ell)+\beta||u-u_{\ell}||^2_{\mathcal{X}},
\end{align*}
then $\{\xi_\ell\}_{\ell\in\mathbb{N}}$ also satisfies
\begin{align*}
\xi_{\ell+k}^2\le C q^k \xi^2_\ell
\end{align*}
with $q= exp(q_{\log}/N_0)<1$ and $C=C_{mon}exp(-q_{\log})$.
\end{lemma}
\begin{proof}
Proof is identical  to \cite[Lemma 6]{feischl2022inf}.  Specifically, our aim is to prove the following estimate
\begin{align}\label{est11}
\xi^2_{\ell+k}\le \left(\Pi_{j=1}^{k}(1-D(j)^{-1})\right)\sum_{j=\ell}^{\ell+k}\xi_{j}^2,\quad \forall k,\ell\in\mathbb{N}.
\end{align}
where 
\begin{align*}
  \xi^2_\ell=\eta^2(\cT_\ell)+\beta||u-u_{\ell}||^2_{\mathcal{X}}.
\end{align*}
The mathematical induction is used to prove the above estimate. For $k=0$, it holds.  Next, we assume that it is true for $k=n$ and prove for $k=n+1$.
\begin{align}\label{est12}
\xi^2_{\ell+k+1}&\le \left(\Pi_{j=1}^{k}(1-D(j)^{-1})\right)\sum_{j=\ell+1}^{\ell+k+1}\xi_{j}^2,\nonumber\\
&\le \left(\Pi_{j=1}^{k}(1-D(j)^{-1})\right)\sum_{j=\ell}^{\ell+k+1}\xi_{j}^2-\xi_{\ell}^2.
\end{align}
Using  (\ref{quasilem11}) gives the estimate (\ref{est11}). Applying (\ref{quasilem11}) implies that
\begin{align}\label{est13}
\xi^2_{\ell+k+1}&\le \left(D(k)\Pi_{j=1}^{k}(1-D(j)^{-1})\right)\xi_{\ell}^2.
\end{align}
Using quasi-monotonicity of $\xi_\ell$ gives
\begin{align*}
\xi_{\ell+k}^2\le \xi_{\ell+aN_0+b}^2&\le q_0^a \xi_{\ell+b}^2%
\le C_{mon}q_0^a \xi_{\ell+b}^2%
\le \frac{C_{mon}q_0^{k/N_0}} {q}\xi_{\ell+b}^2,\\%
\end{align*}
where $k=aN_0+b$ with $a,b\in\mathbb{N}$, $b<N_0$ and $a>k/N_0-1$.
\end{proof}

We are now ready to collect all results and to prove Theorem~\ref{optthm}.

\begin{proof}[Proof of Theorem~\ref{optthm}]
The reliability result, as presented in Lemma~\ref{lem-lin11}, is a direct consequence of the discrete reliability equation~\eqref{eq:drel} and the mesh density. Additionally, the reduction in the estimator outlined in Lemma~\ref{lem-lin11} is derived from the stability exhibited on non-refined elements~\eqref{eq:stabnonref}, coupled with the reduction property on refined elements~\eqref{eq:redprop}. Lastly, the quasi-monotonicity condition specified in Lemma~\ref{lem-lin12} is established through the stability on non-refined elements~\eqref{eq:stabnonref}, the reduction property on refined elements~\eqref{eq:redprop}, and the discrete reliability equation~\eqref{eq:drel}.

Utilizing Remark~\ref{remark1} in conjunction with Lemmas~\ref{lem-lin11} through~\ref{lem-lin12} leads to the conclusion that 
\[
\xi_{\ell+k}^2\le Cq^k \xi_{\ell}^2, \quad \forall \ell,k\in\mathbb{N},
\]
where $q\in(0,1)$ and $C>0$. It's worth noting that all pivotal components referenced in~\cite[Lemma 4.12]{carstensen2014axioms} 
have been rigorously demonstrated to affirm the optimality of D\"{o}rfler marking.
 Consequently, with reference to~\cite[Lemma 4.14]{carstensen2014axioms} and~\cite[Lemma 4.15]{carstensen2014axioms}, the rate optimality of \textbf{Algorithm1} is confirmed.
\end{proof}
\section{Numerical results}\label{sec5}
This section is devoted to discuss the computational experiments which demonstrate 
 the efficacy of the proposed adaptive FEM. All results are computed using FENICS \cite{mardal_book12,AlnaesEtal2015}. The following standard adaptive loop discussed in \textbf{Algorithm 1}
is used to compute the numerical experiments:
\[\mbox{SOLVE}\rightarrow\mbox{ESTIMATE}\rightarrow\mbox{MARK}\rightarrow\mbox{REFINE}\]
To solve the discrete system at each refinement level $\ell$, we employ the SLEPcEigenSolver. D\"{o}fler marking, guided by the criterion
\[\theta\sum_{\mathcal{T}_\ell} \eta_\ell^2 \le \sum_{\mathcal{M}_\ell} \eta_\ell^2,\]
 is used to strategically select elements for refinement, where $\theta\in (0,1)$. The computation of discrete eigenvalues is achieved using Taylor-Hood $P2-P1$ elements.

\subsection{Lshaped-domain}
In this instance, we tackle the Stokes eigenvalue problem (\ref{stokeseig}) within the confines of an L-shaped domain, denoted as $\Omega=(-1,1)^2\setminus [0,1]^2$. The benchmark value for the primary eigenvalue is $\lambda= 32.13269465$, as stipulated in \cite{gedicke2020divergence}.
The error profile for the first eigenvalue is presented in Figure \ref{Lshaped-1}. Notably, the decay of the error, expressed as $\sqrt{|\lambda-\lambda_\ell|}$ between the discrete eigenvalue and the referenced first eigenvalue, follows a rate of approximately $0.27$ with a uniformly refined mesh. Figure \ref{Lshaped-1} illustrates that the slope of the eigenvalue error aligns with that of the estimator $\eta_\ell$, indicating the reliability and efficiency of our estimator. Simultaneously, employing an adaptively refined mesh results in an error decay rate of about $1$ for $\sqrt{|\lambda-\lambda_\ell|}$.
The velocity and pressure profiles are depicted in Figure \ref{Lshaped-1v1}, while Figure \ref{Lshaped-1m1} showcases a sequence of adaptively refined meshes at different levels. These mesh illustrations reveal substantial refinements near the singularity, validating the numerical efficacy of our adaptive scheme.

\begin{figure}[htbp]
\begin{center}
\includegraphics[width=10.5cm,height=7.5cm]{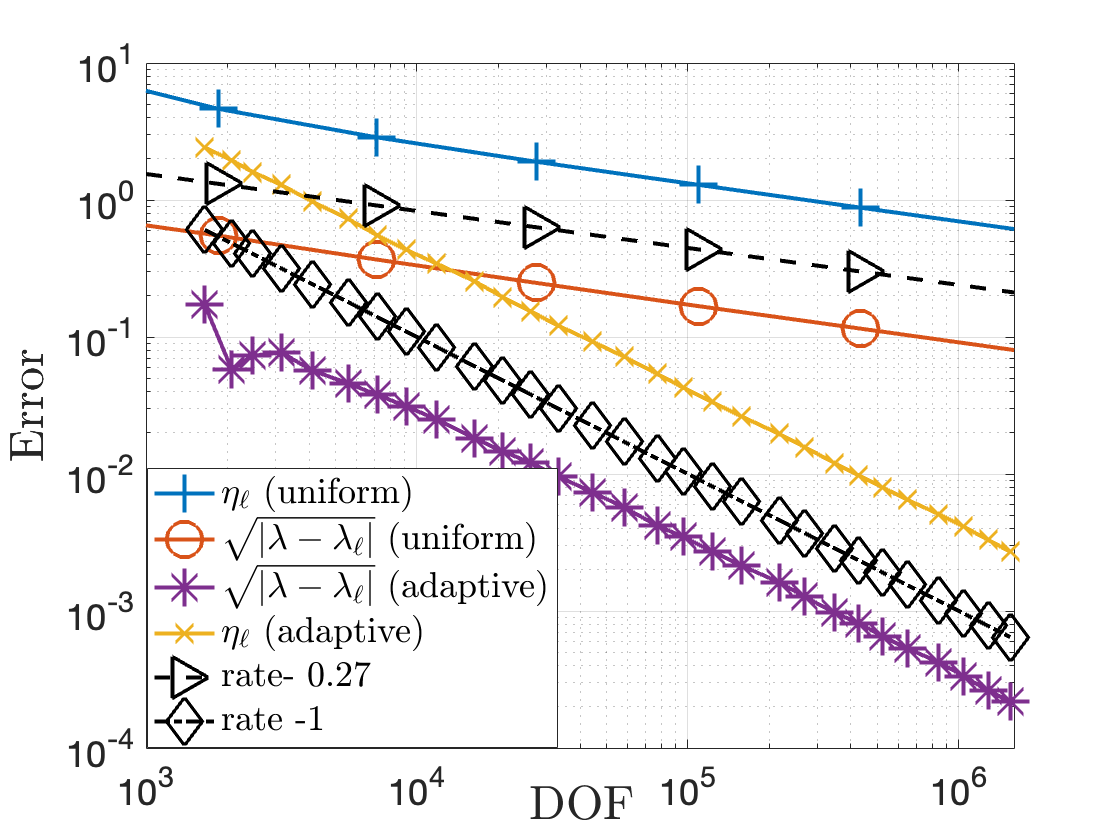}
\caption{Convergence results for Lshaped-domain}
\label{Lshaped-1}
\end{center}
\end{figure}

\begin{figure}[htbp]
\begin{center}
\subfigure[]{
\includegraphics[width=4.5cm,height=4.5cm]{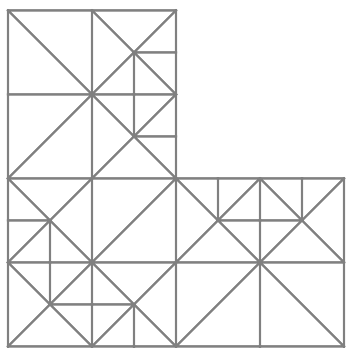}
\label{lshaped-m1}
}
\subfigure[]{
\includegraphics[width=4.5cm,height=4.5cm]{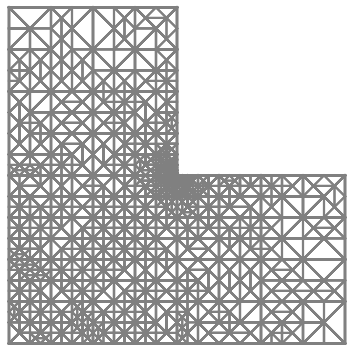}
\label{lshaped-m2}
}
\subfigure[]{
\includegraphics[width=4.5cm,height=4.5cm]{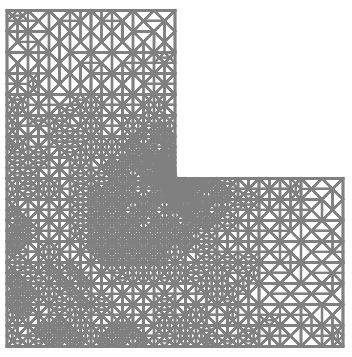}
\label{lshaped-m3}
}
\subfigure[]{
\includegraphics[width=4.5cm,height=4.5cm]{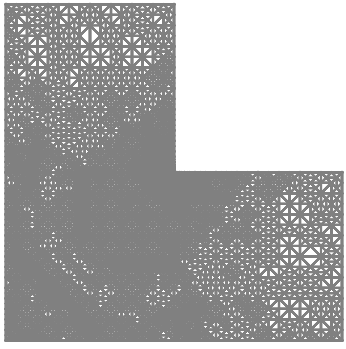}
\label{lshaped-m4}
}
\subfigure[]{
\includegraphics[width=4.5cm,height=4.5cm]{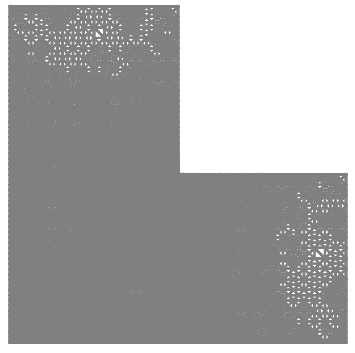}
\label{lshaped-m5}
}
\caption{Adaptively refined meshes for Lshaped-domain with \subref{lshaped-m1} 151, DOF, \subref{lshaped-m2} 5571 DOF, \subref{lshaped-m3} 11836DOF, \subref{lshaped-m4} 26286 DOF, and \subref{lshaped-m5} 58291 DOF.}
\label{Lshaped-1m1}
\end{center}
\end{figure}
\begin{figure}[htbp]
\begin{center}
\subfigure[]{
\includegraphics[width=5.5cm,height=4.5cm]{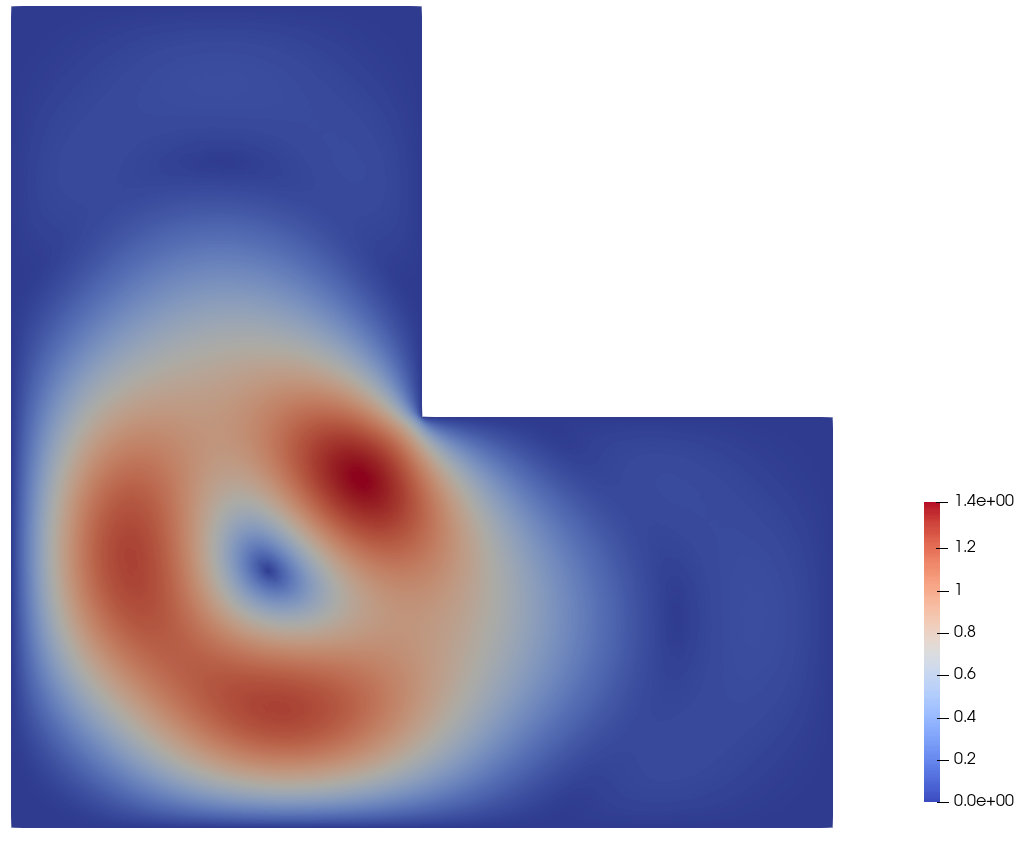}
\label{lshaped-vel}
}
\subfigure[]{
\includegraphics[width=5.5cm,height=4.5cm]{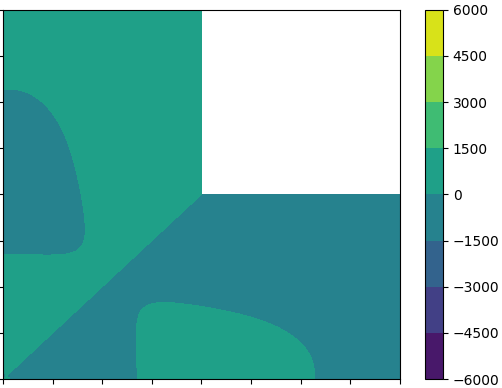}
\label{lshaped-pre}
}
\caption{\subref{lshaped-vel} Velocity profile, \subref{lshaped-pre} Pressure profile,}
\label{Lshaped-1v1}
\end{center}
\end{figure}
\subsection{Slit-domain}
In this illustration, we address the resolution of the Stokes eigenvalue problem (\ref{stokeseig}) within the Slit-domain, denoted as $\Omega=(0,1)^2\setminus [0,1]$. The established benchmark for the primary eigenvalue is $\lambda= 29.9168629$, as documented in \cite{gedicke2020divergence}.
The error profile for the first eigenvalue is depicted in Figure \ref{slit}. Notably, the decay of the error, expressed as $\sqrt{|\lambda-\lambda_\ell|}$ between the discrete eigenvalue and the referenced first eigenvalue, follows a rate of approximately $0.25$ with a uniformly refined mesh. Figure \ref{slit} illustrates that the slope of the eigenvalue error aligns with that of the estimator $\eta_\ell$, signifying the reliability and efficiency of our estimator. Simultaneously, utilizing an adaptively refined mesh results in an error decay rate of about $1$ for $\sqrt{|\lambda-\lambda_\ell|}$.
The velocity and pressure profiles are showcased in Figure \ref{slit-v1}, while Figure \ref{slit1m1} displays a sequence of adaptively refined meshes at different levels. These mesh depictions reveal substantial refinements near the singularity, confirming the anticipated effectiveness of our adaptive scheme.

\begin{figure}[htbp]
\begin{center}
\includegraphics[width=10.5cm,height=7.5cm]{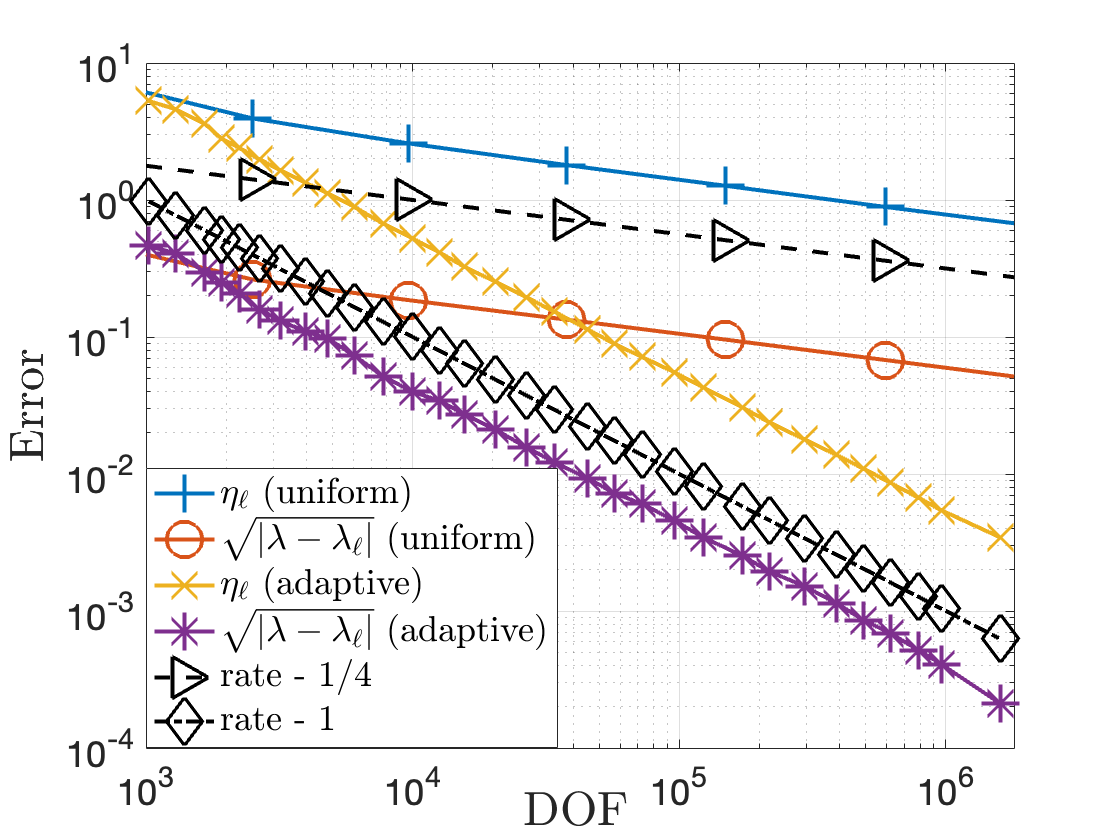}
\caption{Convergence results for Slit-domain}
\label{slit}
\end{center}
\end{figure}
\begin{figure}[htbp]
\begin{center}
\subfigure[]{
\includegraphics[width=4.5cm,height=4.5cm]{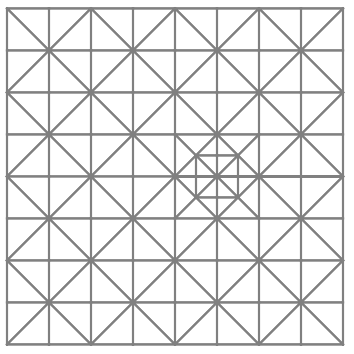}
\label{slit-m1}
}
\subfigure[]{
\includegraphics[width=4.5cm,height=4.5cm]{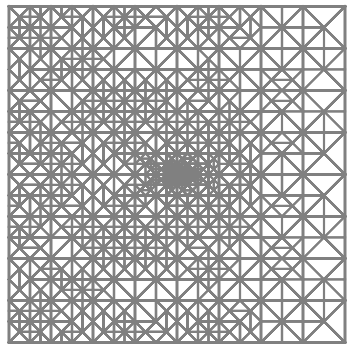}
\label{slit-m2}
}
\subfigure[]{
\includegraphics[width=4.5cm,height=4.5cm]{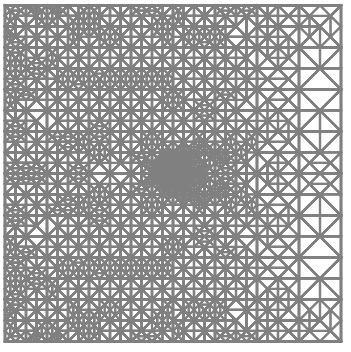}
\label{slit-m3}
}
\subfigure[]{
\includegraphics[width=4.5cm,height=4.5cm]{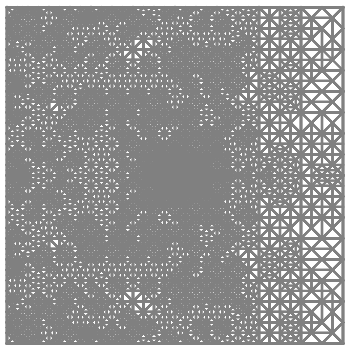}
\label{slit-m4}
}
\caption{Adaptively refined meshes for Slit-domain with \subref{slit-m1} 679, DOF, \subref{slit-m2} 6047 DOF, \subref{slit-m3} 12636 DOF, and \subref{slit-m4} 26718 DOF.}
\label{slit1m1}
\end{center}
\end{figure}
\begin{figure}[htbp]
\begin{center}
\subfigure[]{
\includegraphics[width=5.5cm,height=4.5cm]{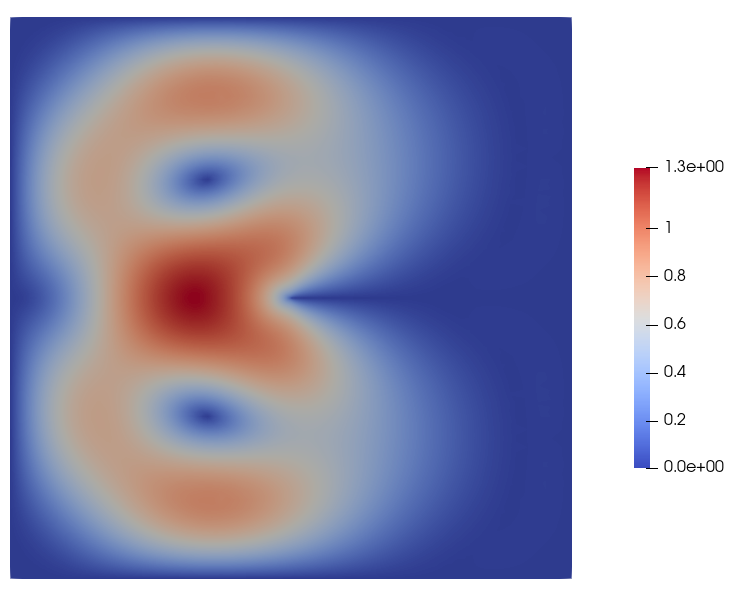}
\label{slit-vel}
}
\subfigure[]{
\includegraphics[width=5.5cm,height=4.5cm]{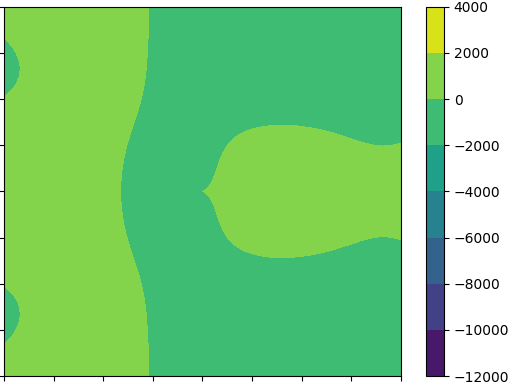}
\label{slit-pre}
}
\caption{\subref{slit-vel} Velocity profile, \subref{slit-pre} Pressure profile.}
\label{slit-v1}
\end{center}
\end{figure}

\subsection{The Fichera corner}
In this example, we address the Stokes eigenvalue problem (\ref{stokeseig}) in the context of the Fichera corner, represented as $\Omega= \Omega_1\setminus \bar{\Omega}_2$, where $\Omega_1= (0,0.8)\times(0.0,1)\times(0.0,1.2)$ and  $\Omega_2= (0.0,0.4)\times (0.0,0.5)\times (0.0,0.6)$. For this particular example, we designate $\theta$ as $0.1$. The graphical representation in Figure \ref{3dl-error} demonstrates that the estimator $\eta_\ell$ exhibits a decay rate of approximately $2/3$ when utilizing adaptive meshes.
Figure \ref{3dl-v} presents the velocity and pressure profiles, while Figure \ref{3dl-v}\subref{3dl-vel11} showcases the adaptively refined mesh containing $1009839$ degrees of freedom (DOFs). These mesh visualizations highlight significant refinements in proximity to the singularity, affirming the numerical effectiveness of our adaptive scheme.
\begin{figure}[htbp]
\begin{center}
\includegraphics[width=10.5cm,height=7.5cm]{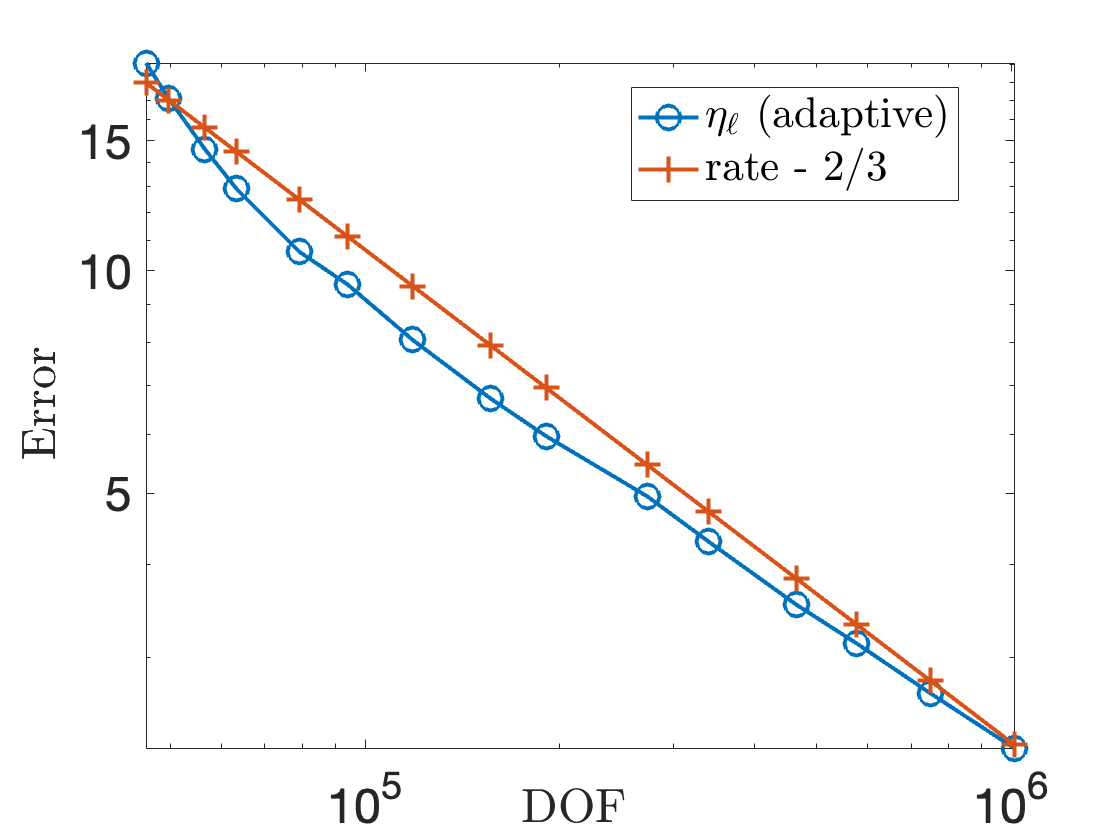}
\caption{Convergence results for the Fichera corner}
\label{3dl-error}
\end{center}
\end{figure}
\begin{figure}[htbp]
\begin{center}
\subfigure[]{
\includegraphics[width=6cm,height=5.5cm]{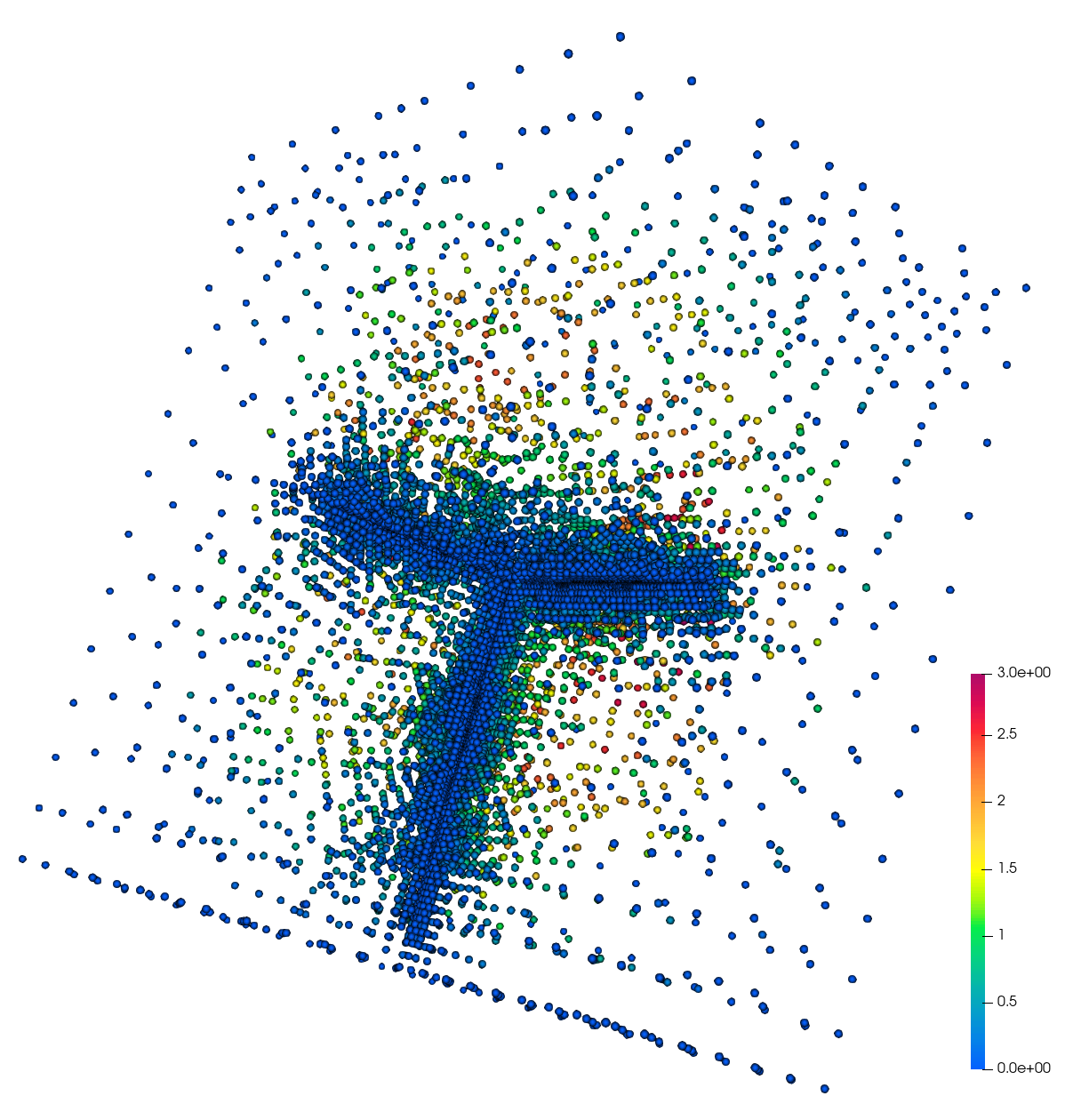}
\label{3dl-vel}
}
\subfigure[]{
\includegraphics[width=6cm,height=5.5cm]{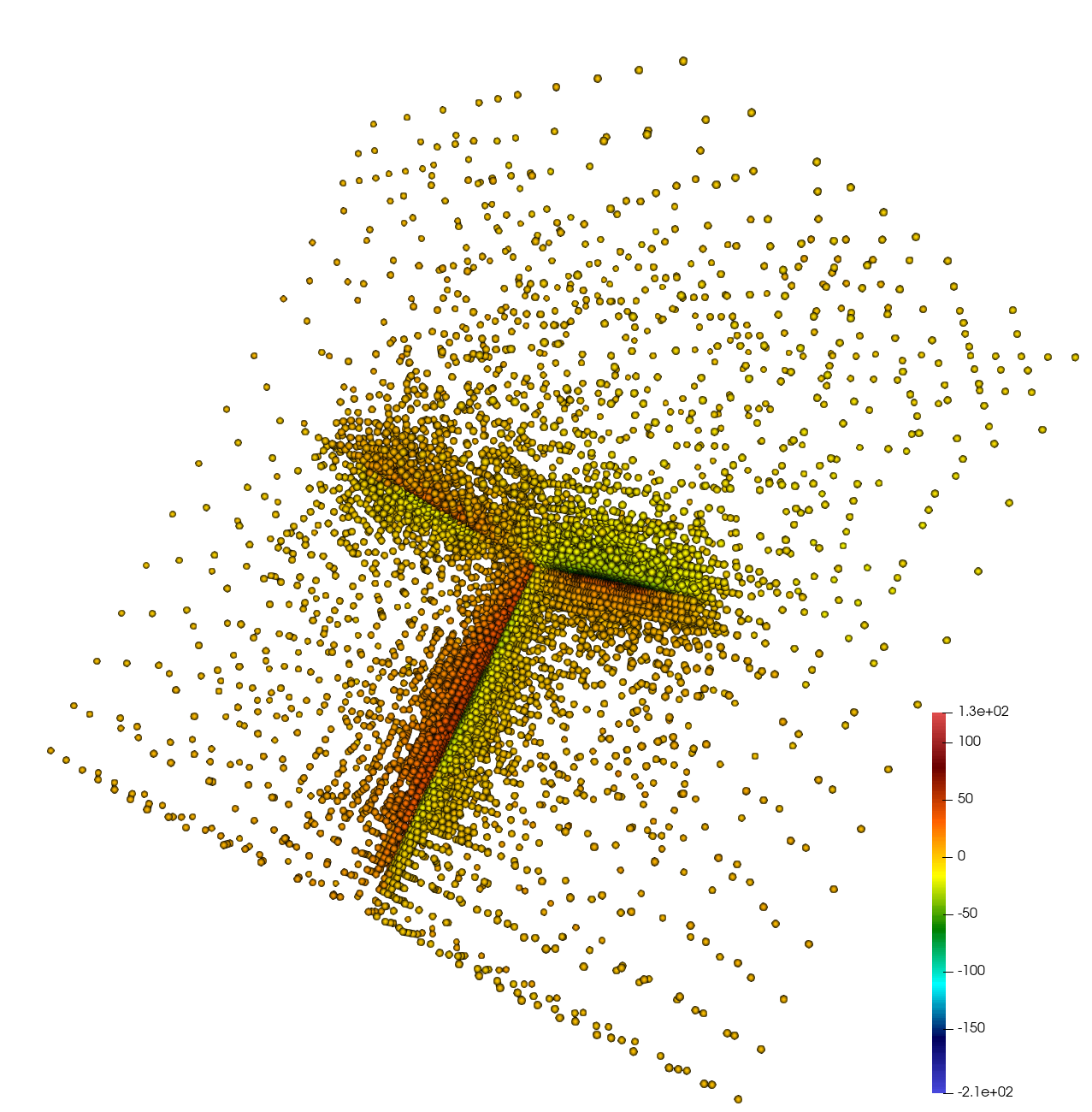}
\label{3dl-pre}
}
\subfigure[]{
\includegraphics[width=7cm,height=6cm]{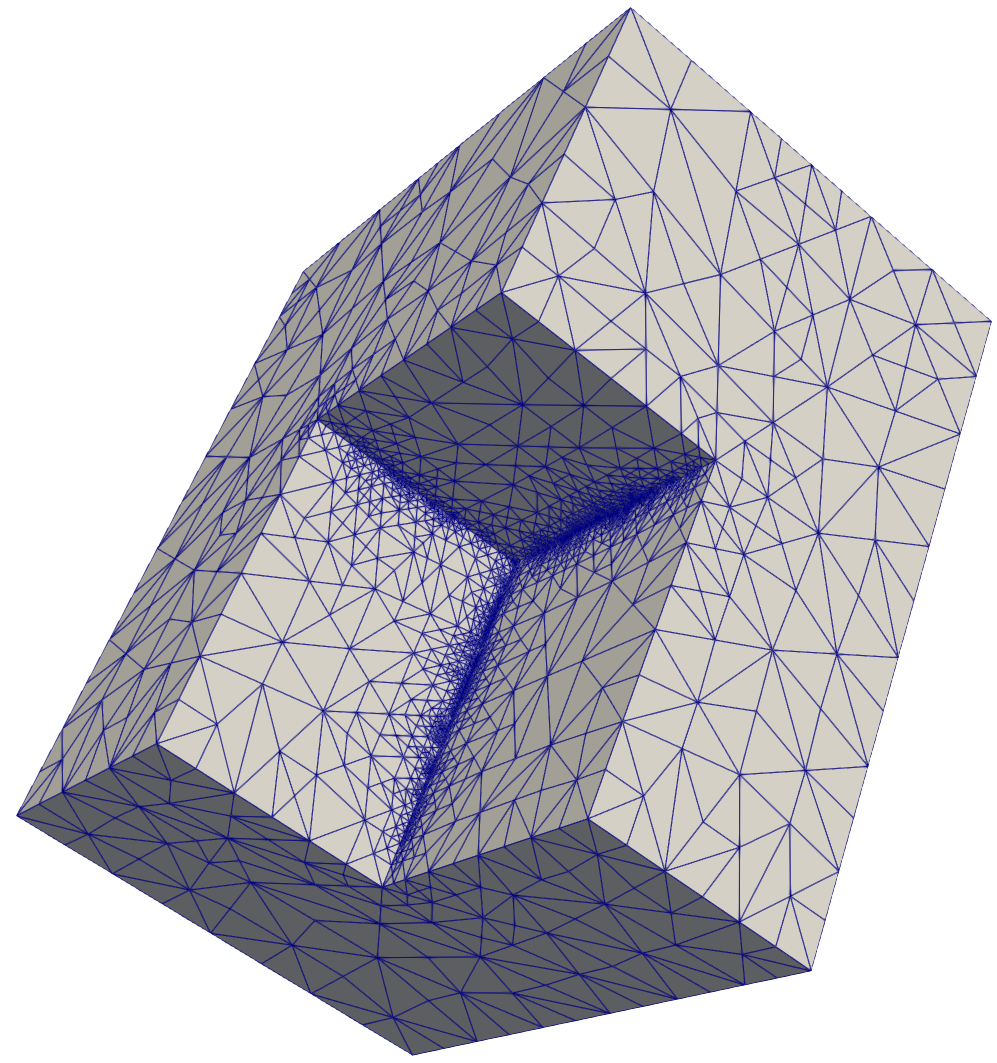}
\label{3dl-vel11}
}
\caption{\subref{3dl-vel} Velocity profile, \subref{3dl-pre} Pressure profile, \subref{3dl-vel11} Adaptively refined meshes for the Fichera corner with 1009839 DOFs.}
\label{3dl-v}
\end{center}
\end{figure}

\section*{Conclusion}
This paper addressed the application of adaptive mixed FEM to Stokes eigenvalue problems. We achieved the following key results:
\begin{itemize}
\item{\textbf{Rate optimality:}} Established the optimal convergence rate for  \textbf{Algorithm 1}, demonstrating its efficiency in approximating eigenvalues.
\item{\textbf{Novel estimator:}} Proposed a new estimator specifically tailored for Stokes eigenvalue problems. This estimator exhibits local equivalence to the established work of Lovadina et al. \cite{lovadina2009posteriori}.
\item{\textbf{Theoretical rigor:}} Provided rigorous proofs for quasi-orthogonality and discrete reliability, solidifying the theoretical foundation of our approach.
\item{\textbf{Numerical validation:}} Confirmed the optimality result through computational experiments, showcasing the practical effectiveness of our method.
\end{itemize}
 These findings contribute significantly to the field of adaptive FEM for eigenvalue problems, offering a robust and efficient framework for tackling Stokes eigenvalue analysis.
  \bibliographystyle{amsplain} 
\bibliography{BibFile-stokeseigen}
\end{document}